\documentclass[12pt]{amsart}
\usepackage{amssymb,latexsym}
\usepackage{enumerate}
\usepackage{dsfont}
\usepackage{a4wide}
\usepackage{mathrsfs}
\usepackage{dsfont}
\usepackage[cp1250]{inputenc}
\usepackage[T1]{fontenc}
\usepackage{tabularx}
\usepackage{multirow}
\usepackage{scalerel}
\usepackage{etoolbox}
\patchcmd{\section}{\normalfont}{\normalfont\Large}{}{}
\patchcmd{\section}{\scshape}{\bfseries}{}{}
\makeatletter
\renewcommand{\@secnumfont}{\bfseries}
\makeatother
\usepackage{lmodern}
\usepackage{hyperref}
\hypersetup{hypertexnames=false}
\let\originalforall=\forall
\renewcommand{\forall}{\mathop{\vcenter{\hbox{\Large$\originalforall$}}}}

\let\originalexists=\exists
\renewcommand{\exists}{\mathop{\vcenter{\hbox{\Large$\originalexists$}}}}
\newcommand{\dziub}{\hspace{-4 pt}\vspace{6 pt}\hat{\phantom{i}}}
\newtheorem{thm}{Theorem}[section]
\newtheorem{cor}[thm]{Corollary}
\newtheorem{lem}[thm]{Lemma}
\newtheorem{prop}[thm]{Proposition}

\newtheorem{prob}[]{Problem}
\newtheorem{de}[thm]{Definiton}
\newtheorem{rem}[thm]{Remark}
\numberwithin{equation}{section}

\date{\today}
\title{Inversion problem in measure and Fourier--Stieltjes algebras}
\author{Przemys\l aw Ohrysko}
\address{Chalmers University of Technology and the University of Gothenburg}
\email{p.ohrysko@gmail.com}
\author{Mateusz Wasilewski}
\address{KU Leuven}
\email{mateusz.wasilewski@kuleuven.be}
\thanks{PO was supported by foundations managed by The Royal Swedish Academy of Sciences. MW was partially supported by  National Science Centre (NCN) grant 2016/21/N/ST1/02499, long term structural funding -- Methusalem grant of the Flemish Government -- and by European Research Council Consolidator Grant 614195 RIGIDITY.}
\begin{document}
\baselineskip=21pt
\begin{abstract}
In this paper we study the inversion problem in measure and Fourier--Stieltjes algebras from qualitative and quantitative point of view extending the results obtained by N. Nikolski in \cite{nik}.
\end{abstract}
\subjclass[2010]{Primary 43A05; Secondary 43A30.}

\keywords{Inverse, Measure Algebras, Fourier--Stieltjes algebras.}
\maketitle

\section{Introduction}
We are going to collect first some basic facts from Banach algebra theory and harmonic analysis in order to fix the notation (our main reference for Banach algebra theory is \cite{z}, for harmonic analysis check \cite{r}). For a commutative unital Banach algebra $A$, the Gelfand space of $A$ (the set of all multiplicative-linear functionals endowed with weak$^{\ast}$-topology) will be denoted by $\triangle(A)$ and the Gelfand transform of an element $x\in A$ is a surjection $\widehat{x}:\triangle(A)\rightarrow\sigma(x)$ defined by the formula: $\widehat{x}(\varphi)=\varphi(x)$ for $\varphi\in\triangle(A)$, where $\sigma(x):=\{\lambda\in\mathbb{C}:\mu-\lambda\delta_{0}\text{ is not invertible}\}$ is the spectrum of an element $x$.
Let $G$ be a locally compact Abelian group with its unitary dual $\widehat{G}$ and let $M(G)$ denote the Banach algebra of all complex-valued Borel regular measures equipped with the convolution product and the total variation norm. It is also  a $^{\ast}$-algebra with involution $\mu\rightarrow\widetilde{\mu}$ defined for any Borel set $E\subset G$ by $\widetilde{\mu}(E):=\overline{\mu(-E)}$. A measure $\mu\in M(G)$ is called hermitian, if $\mu=\widetilde{\mu}$ or, equivalently, its Fourier-Stieltjes transform is real-valued. The Fourier-Stieltjes transform will be treated as a restriction of the Gelfand transform to $\widehat{G}$. Note that we have a direct sum decomposition $M(G)=M_{c}(G)\oplus M_{d}(G)$ where $M_{c}(G)$ is the ideal of continuous (non-atomic) measures and $M_{d}(G)$ is the subalgebra of discrete (atomic) measures. For $\mu\in M(G)$ we will write $\mu=\mu_{c}+\mu_{d}$ with $\mu_{c}\in M_{c}(G)$ and $\mu_{d}\in M_{d}(G)$.

We recall the problem investigated by N. Nikolski in \cite{nik}, which will be our main point of interest.
\begin{prob}\label{pier}
Let $\mu\in M(G)$ satisfy $\|\mu\|\leq 1$ and suppose that $\inf_{\gamma\in\widehat{G}}|\widehat{\mu}(\gamma)|=\delta>0$. What is the minimal value of $\delta_{0}>0$ such that for any $\delta>\delta_{0}$ the measure $\mu$ is automatically invertible? What can be said about the norm of the inverse?
\end{prob}
It is clear that $\delta_{0}<\frac{1}{2}$ is not enough as the example of a measure $\mu:=\frac{1}{2}\delta_{0}+\frac{1}{2}\nu$ where $\nu$ is a probability measure with non-negative Fourier-Stieltjes transform satisfying $\sigma(\nu)=\overline{\mathbb{D}}$ shows (here $\delta_{0}$ is the Dirac delta at the point $0$). On the other hand, it was proved in \cite{nik} that any $\delta>\frac{1}{\sqrt{2}}$ does the job and moreover the norm of the inverse is bounded by $(2\delta^{2}-1)^{-1}$. For the readers convenience we will reprove this result now with a slightly simpler (but less general) approach. Let $\mu\in M(G)$ satisfy $\|\mu\|\leq 1$ and $\inf_{\gamma\in\widehat{G}}|\widehat{\mu}(\gamma)|=\delta>\frac{1}{\sqrt{2}}$. Then, by the generalization of Wiener lemma (see 5.6.9 in \cite{r}):
\begin{equation*}
(\mu\ast\widetilde{\mu})(\{0\})=\sum_{x\in G}|\mu(\{x\})|^{2}=\lim_{\alpha}\int_{\Gamma}\widehat{f_{\alpha}}(\gamma)|\widehat{\mu}(\gamma)|^{2}d\gamma\geq\delta^{2}>\frac{1}{2},
\end{equation*}
where $\{f_{\alpha}\}$ is a system of continuous positive-definite functions with compact supports subordinated to a neighborhood base $\{V_{\alpha}\}$ of $0$ with $f_{\alpha}(0)=1$.
The second ingredient is a version of Lemma 1.4.3 from \cite{nik} adapted to our situation.
\begin{lem}\label{latw}
Let $\mu\in M(G)$ and let $\mu=\lambda\delta_{0}+\nu$ where $\nu(\{0\})=0$ and $\frac{1}{2}<\delta\leq|\lambda|\leq \|\mu\|\leq 1$. Then $\mu$ is invertible and $\|\mu^{-1}\|\leq\frac{1}{2\delta-1}$.
\end{lem}
By the lemma, $\|(\mu\ast\widetilde{\mu})^{-1}\|\leq\frac{1}{2\delta^{2}-1}$ and in order to finish the argument one has only to observe that $\mu^{-1}=(\mu\ast\widetilde{\mu})^{-1}\ast\widetilde{\mu}$.

However, the minimal value of $\delta_{0}>0$ seems to be unknown and the question on improving the aforementioned bounds is stated in \cite{nik}.

The first aim of this paper is to give a proof of the following fact (Theorem \ref{glop}): if $\mu\in M(G)$ satisfies $\|\mu\|\leq 1$ and $\inf_{\gamma\in\widehat{G}}|\widehat{\mu}(\gamma)|>\frac{1}{2}$ then $\mu$ is invertible, which is the final solution of the problem (proving $\delta_{0}=\frac{1}{2}$) -- in view of the previous discussion the result is sharp.

In the next part of this article we attack the problem of estimating the norm of the inverse. In addition to the upper bounds mentioned above, the paper \cite{nik} contains a result on the lower bound. We will not reproduce the whole discussion here but what is crucial for us is that one cannot hope for a bound better than $(2\delta-1)^{-1}$ and the norm-controlled inversion cannot hold for any $\delta\leq\frac{1}{2}$ for any infinite locally compact Abelian group. The most important result is Theorem \ref{nies} which states that under additional group-theoretic assumption the norm-controlled inversion holds true for $\delta>\frac{-1+\sqrt{33}}{8}\simeq 0,593$ - this improves the result of N. Nikolski ($\delta>\frac{1}{\sqrt{2}}\simeq 0,707$). Later, we show in Theorem \ref{nornz} that for measures with discrete parts supported on independent subsets of locally compact Abelian groups consisting of elements of infinite order we obtain the optimal bound. The section is concluded with a discussion of another special case -- real measures on groups of exponent two.
\\
The last, independent part of the paper is devoted to the study of analogous problems for Fourier-Stieltjes algebras. Particularly, we prove in Theorem \ref{fsa} that the qualitative inversion problem has a positive solution for any $\delta>\frac{1}{2}$. We also prove (Theorem \ref{jprzez}) the norm-controlled inversion for $\delta>\frac{1}{\sqrt{2}}$ in this context.

\section{Measure algebras}
\subsection{Qualitative result}
\begin{lem}\label{latwy}
Let $\mu\in M(G)$ satisfy $\|\mu\|\leq 1$ and suppose that $\inf_{\gamma\in\Gamma} |\widehat{\mu_{d}}(\gamma)|>\frac{1}{2}$. Then $\mu$ is invertible.
\end{lem}
\begin{proof}
Recalling that $\sigma(\mu_{d})=\sigma_{M_{d}(G)}(\mu_{d})$ and $\triangle(M_{d}(G))=\triangle(l^{1}(G))=b\widehat{G}$ (the Bohr compactfication of $\widehat{G}$) we get
\begin{equation}\label{dys}
\forall_{\varphi\in\triangle(M(G))}|\widehat{\mu_{d}}(\varphi)|>\frac{1}{2}.
\end{equation}
Of course, $\frac{1}{2}< r(\mu_{d})\leq \|\mu_{d}\|$ so
\begin{equation}\label{con}
\|\mu_{c}\|=\|\mu\|-\|\mu_{d}\|<\frac{1}{2}.
\end{equation}
Suppose now, towards the contradiction, that $\mu$ is not invertible. Then $0\in\sigma(\mu)$ and there exists $\varphi_{0}\in\triangle(M(G))$ for which $\varphi_{0}(\mu)=0$. This gives $\varphi_{0}(\mu_{d})=-\varphi_{0}(\mu_{c})$ implying that $|\varphi_{0}(\mu_{d})|=|\varphi_{0}(\mu_{c})|$. But by (\ref{dys}) we have $|\varphi_{0}(\mu_{d})|>\frac{1}{2}$ and by (\ref{con}) we get $|\varphi_{0}(\mu_{c})|\leq r(\mu_{c})\leq \|\mu_{c}\|<\frac{1}{2}$ which is a contradiction.
\end{proof}
Here we need to cite the theorem of I. Glicksberg and I. Wik (see (2) in \cite{gw}).
\begin{thm}\label{dysk}
Let $\mu\in M(G)$. Then $\widehat{\mu_{d}}(\widehat{G})\subset\overline{\widehat{\mu}(\widehat{G})}$.
\end{thm}
A straightforward corollary of this theorem is the following observation.
\begin{cor}\label{dysdu}
Let $\mu\in M(G)$. Then $\inf_{\gamma\in\widehat{G}}|\widehat{\mu_{d}}(\gamma)|\geq \inf_{\gamma\in\widehat{G}}|\widehat{\mu}(\gamma)|$.
\end{cor}
We are prepared now to ultimately solve the inversion problem.
\begin{thm}\label{glop}
Let $\mu\in M(G)$ satisfy $\|\mu\|\leq 1$ and $\inf_{\gamma\in\widehat{G}}|\widehat{\mu}(\gamma)|>\frac{1}{2}$. Then $\mu$ is invertible.
\end{thm}
\begin{proof}
By Corollary \ref{dysdu} we have $\inf_{\gamma\in\widehat{G}}|\widehat{\mu_{d}}(\gamma)|\geq\inf_{\gamma\in\widehat{G}}|\widehat{\mu}(\gamma)|>\frac{1}{2}$. To finish the proof we just need to apply Lemma \ref{latwy}.
\end{proof}
\subsection{Quantitative results}
\subsubsection{The general case}
In this section we will attack the problem of improving the constant $\delta_{0}$ giving the following assertion: if $\mu\in M(G)$ satisfy $\|\mu\|\leq 1$ and $\inf_{\gamma\in\widehat{G}}|\widehat{\mu}(\gamma)|\geq\delta>\delta_{0}$ then $\|\mu^{-1}\|$ can be bounded as a function of $\delta$ (norm-controlled inversion). As it was explained in the introduction, if $G$ is non-discrete, then $\delta_{0}\geq\frac{1}{2}$. In addition, a non-trivial argument from \cite{nik} shows that the same statement holds true for arbitrary infinite discrete groups. Let us also recall that $\delta_{0}\leq\frac{1}{\sqrt{2}}$.
\\
We start with an elementary lemma.
\begin{lem}\label{sumki}
Let $(x_{n})_{n=1}^{\infty}$ be a non-increasing sequence of non-negative numbers satisfying
\begin{equation}\label{sum}
\sum_{n=1}^{\infty}x_{n}\leq 1,
\end{equation}
\begin{equation}\label{kwad}
\sum_{n=1}^{\infty}x_{n}^{2}\geq\delta^{2},\text{ where }1\geq\delta>\frac{1}{2}.
\end{equation}
Then the following inequalities hold true
\begin{equation}\label{pierw}
x_{1}\geq\delta^{2},
\end{equation}
\begin{equation}\label{sd}
x_{1}+x_{2}\geq\delta.
\end{equation}
\end{lem}
\begin{proof}
In order to prove (\ref{pierw}) we use consecutively (\ref{kwad}), monotonicity of the sequence $(x_{n})_{n=1}^{\infty}$ and then (\ref{sum})
\begin{equation*}
\delta^{2}\leq \sum_{n=1}^{\infty}x_{n}^{2}\leq x_{1}\sum_{n=1}^{\infty}x_{n}\leq x_{1}.
\end{equation*}
For the second inequality we first observe that by (\ref{kwad}) and the monotonicity:
\begin{equation*}
\delta^{2}\leq \sum_{n=1}^{\infty}x_{n}^{2}=x_{1}^{2}+\sum_{n=2}^{\infty}x_{n}^{2}\leq x_{1}^{2}+x_{2}\sum_{n=2}^{\infty}x_{n}.
\end{equation*}
Now, we apply (\ref{sum}) which gives
\begin{equation}\label{glk}
\delta^{2}\leq x_{1}^{2}+x_{2}(1-x_{1}).
\end{equation}
If $x_{1}\geq\delta$ then inequality (\ref{sd}) is obvious, so we are allowed to assume that $x_{1}< \delta$ and then (\ref{glk}) can be rewritten as
\begin{equation}\label{ddoln}
x_{2}\geq\frac{\delta^{2}-x_{1}^{2}}{1-x_{1}}.
\end{equation}
Thus, it is enough to justify
\begin{equation}\label{nie}
x_{1}+\frac{\delta^{2}-x_{1}^{2}}{1-x_{1}}\geq\delta\text{ for }x_{1}\in [\delta^{2},\delta].
\end{equation}
This is equivalent to
\begin{equation*}
2x_{1}^{2}-(1+\delta)x_{1}+\delta-\delta^{2}\leq 0\text{ for }x_{1}\in [\delta^{2},\delta].
\end{equation*}
This is a quadratic inequality with the discriminant $(3\delta-1)^{2}>\frac{1}{4}$. Elementary calculations lead to $x_{1}\in [\frac{1-\delta}{2},\delta]$ and therefore the assertion will be proved as long as $\frac{1-\delta}{2}\leq\delta^{2}$. However, the last inequality simply follows from the assumption $\delta>\frac{1}{2}$ establishing (\ref{sd}) and finishing the proof of the lemma.
\end{proof}
This lemma can be applied to obtain the following proposition.
\begin{prop}\label{sumadw}
Let $G$ be a locally compact Abelian group and let $\mu\in M(G)$ satisfy $\|\mu\|\leq 1$ and $|\widehat{\mu}(\gamma)|\geq\delta>\frac{1}{2}$ for every $\gamma\in\widehat{G}$. If
\begin{equation*}
\mu_{d}=\sum_{n=1}^{\infty}a_{n}\delta_{\tau_{n}}\text{ with }|a_{1}|\geq |a_{2}|\geq\ldots\text{ and for some }\tau_{n}\in G,
\end{equation*}
then $|a_{1}|\geq\delta^{2}$ and $|a_{1}|+|a_{2}|\geq\delta$.
\end{prop}
\begin{proof}
Clearly, $\|\mu_{d}\|\leq 1$ which is equivalent to the assumption (\ref{sum}) from Lemma \ref{sumki}.
\\
We can treat $\mu_{d}$ as an element of $l^{2}(G_{d})$ ($G_{d}$ is the original group $G$ equipped with the discrete topology). Then $\widehat{\mu_{d}}\in L^{2}\left(\left(G_{d}\right)\dziub\right)$. It is well-known that $\left(G_{d}\right)\dziub$ is canonically isomorphic to $b\widehat{G}$ -- the Bohr compactification of $\widehat{G}$. By Theorem \ref{dysk} we have $|\widehat{\mu_{d}}(\gamma)|\geq\delta$ for every $\gamma\in\widehat{G}$ and since $\widehat{G}$ is dense in $b\widehat{G}$ we obtain by Parseval's identity:
\begin{equation*}
\sum_{n=1}^{\infty}|a_{n}|^{2}=\int_{b\widehat{G}}|\widehat{\mu_{d}}|^{2}(x)dx\geq\delta^{2} \text{ (here $dx$ is the normalized Haar measure on $b\widehat{G}$)},
\end{equation*}
which proves \eqref{kwad} from Lemma \ref{sumki} and enables us to use the assertion of the lemma to finish the argument.
\end{proof}
In order to proceed further, we need yet another technical lemma.
\begin{lem}\label{rzedy}
Let $G$ be a locally compact Abelian group and for $x\in G$ define
\begin{equation*}
S_{x}=\{\gamma(x):\gamma\in\widehat{G}\}\subset\{z\in\mathbb{C}:|z|=1\}.
\end{equation*}
If $x$ is of order $n<\infty$ then $S_{x}=\mathbb{Z}_{n}$ and if $x$ is of infinite order then $S_{x}$ is a dense subgroup of $\{z\in\mathbb{C}:|z|=1\}$.
\end{lem}
\begin{proof}
In case of element $x\in G$ of finite order $n$ we simply observe that the closed subgroup generated by $x$ is isomorphic to $\mathbb{Z}_{n}$ and as $\mathbb{Z}_{n}\hspace{-3 pt}\widehat{\phantom{i}}=\mathbb{Z}_{n}$ the result follows by the well-known fact that characters on the closed subgroup can be extended to the whole group.
\\
If $x$ is of infinite order we formally verify first that $S_{x}$ is a subgroup of the circle group and since every infinite subgroup of the circle group is dense it is enough to show $\# S_{x}=\infty$. Let $H$ be a closed subgroup of $G$ generated by $x$. By definition, $H$ is a \textit{monothetic group}\footnote{A locally compact Abelian group is called monothetic if it contains a dense homomorphic image of $\mathbb{Z}$.} and by the results from Section 2.3 in \cite{r} we obtain that either $H=\mathbb{Z}$ or $H$ is a compact group whose dual is a subgroup of $\mathbb{T}_{d}$ (the circle group with the discrete topology). The first case is elementary and for the second one we continue as follows: since $H$ is infinite, $\widehat{H}$ is also infinite. Now, if $\widehat{H}$ contains an element of infinite order then we are done, as every character in $\widehat{H}$ is uniquely defined by its action on $x$. Otherwise, $\widehat{H}$ contains elements of arbitrarily high order and again since $x$ topologically generates $H$ we obtain that for a fixed $\gamma_{0}\in\widehat{H}$ the set $\{\gamma_{0}^{k}(x):k\in\mathbb{N}\}\subset S_{x}$ has the same number of elements as the order of $\gamma_{0}$, which finishes the proof.
\end{proof}
We are ready now to prove one of the two main theorems of this section. Note that the group-theoretic assumption from the next theorem is automatically satisfied for some classical groups such as $\mathbb{R}$ and $\mathbb{Z}$.
\begin{thm}\label{nies}
Let $\mu\in M(G)$ satisfy $\|\mu\|\leq 1$ and $|\widehat{\mu}(\gamma)|>\delta>\frac{1}{2}$ for every $\gamma\in\widehat{G}$. Let
\begin{equation*}
\mu_{d}=\sum_{n=1}^{\infty}a_{n}\delta_{\tau_{n}},\text{ }|a_{1}|\geq |a_{2}|\geq\ldots.
\end{equation*}
If the order of element $\tau_{2}-\tau_{1}$ is infinite then
\begin{equation*}
|a_{1}|\geq \frac{1-\delta+\sqrt{17\delta^{2}+6\delta-7}}{4}\geq\frac{3}{2}\delta-\frac{1}{2}\text{ and }
\end{equation*}
\begin{gather*}
\|\mu^{-1}\|\leq\frac{1}{3\delta-2}\text{ for }\delta>\frac{2}{3},\\
\|\mu^{-1}\|\leq\frac{2}{-(1+\delta)+\sqrt{17\delta^{2}+6\delta-7}}\text{ for }\delta>\frac{-1+\sqrt{33}}{8}\simeq 0,593.
\end{gather*}
\end{thm}
\begin{proof}
Since both multiplication by a constant of modulus one and convolution with a Dirac delta do not violate the assumptions of the theorem and do not change the norm of the inverse, we are allowed to work with a measure $\nu:=c\mu\ast\delta_{-\tau_{1}}$ where $ca_{1}=|a_{1}|$ instead of the original measure $\mu$. Then
\begin{equation*}
\nu_{d}=|a_{1}|\delta_{0}+ca_{2}\delta_{\tau_{2}-\tau_{1}}+\rho,\text{ where }\rho:=c\sum_{n=3}^{\infty}a_{n}\delta_{\tau_{n}-\tau_{1}}.
\end{equation*}
As the order of $\tau_{2}-\tau_{1}$ is infinite, basing on Lemma \ref{rzedy}, we are allowed to pick a sequence $(\gamma_{n})_{n=1}^{\infty}\subset\widehat{G}$ such
that
\begin{equation}\label{min}
\widehat{\delta}_{\tau_{2}-\tau_{1}}(\gamma_{n})\cdot ca_{2}=\gamma_{n}(\tau_{2}-\tau_{1})\cdot ca_{2}\xrightarrow[n\rightarrow\infty]{}-|a_{2}|.
\end{equation}
By our assumptions and Theorem \ref{dysk} we have
\begin{equation}\label{fnie}
||a_{1}|+ca_{2}\widehat{\delta}_{\tau_{2}-\tau_{1}}(\gamma_{n})+\widehat{\rho}(\gamma_{n})|\geq\delta.
\end{equation}
Using elementary inequalities, $\|\rho\|=\|\mu_{d}\|-|a_{1}|-|a_{2}|$ and passing with the parameter $n$ to infinity we obtain, by (\ref{min}),
\begin{equation}\label{drugmal}
\|\mu_{d}\|-2|a_{2}|\geq\delta\text{, which is equivalent to }|a_{2}|\leq\frac{\|\mu_{d}\|-\delta}{2}\leq\frac{1-\delta}{2}.
\end{equation}
By Proposition \ref{sumadw} we get $|a_{1}|+|a_{2}|\geq\delta$. Combining this with (\ref{drugmal}) we obtain $|a_{1}|\geq\frac{3}{2}\delta-\frac{1}{2}$ and by Lemma \ref{latw} we get $\|\mu^{-1}\|\leq\frac{1}{3\delta-2}$ for $\delta>\frac{2}{3}$.
\\
In order to get a more refined bound we recall that by (\ref{ddoln}) (from the proof of Lemma \ref{sumki}) and (\ref{drugmal}):
\begin{equation}\label{kwadra}
\frac{1-\delta}{2}\geq\frac{\delta^{2}-|a_{1}|^{2}}{1-|a_{1}|}\text{ or equivalently }|a_{1}|^{2}-\frac{1-\delta}{2}|a_{1}|+\frac{1-\delta}{2}-\delta^{2}\geq 0.
\end{equation}
It is a quadratic inequality with discriminant $D:=\frac{17}{4}\delta^{2}+\frac{3}{2}\delta-\frac{7}{4}$, which is easily shown to be positive as $\delta>\frac{1}{2}$. With the aid of Vi\`{e}te formulas (note that $\frac{1-\delta}{2}-\delta^{2}<0$), we finally get
\begin{equation*}
|a_{1}|\geq\frac{\frac{1-\delta}{2}+\sqrt{D}}{2}.
\end{equation*}
The elementary verification of the inequality $\frac{\frac{1-\delta}{2}+\sqrt{D}}{2}\geq\frac{3}{2}\delta-\frac{1}{2}$ and the condition $\frac{\frac{1-\delta}{2}+\sqrt{D}}{2}>\frac{1}{2}$ iff $\delta>\frac{-1+\sqrt{33}}{8}$ finishes the argument by Lemma \ref{latw}.
\end{proof}
It is worth to state the full formulation for the most classical case of $G=\mathbb{Z}$.
\begin{thm}
Let $f\in A(\mathbb{T})$ satisfy $\|f\|\leq 1$ and $|f(t)|\geq\delta>\frac{1}{2}$ for every $t\in\mathbb{T}$. Then
\begin{gather*}
\left\|\frac{1}{f}\right\|\leq\frac{1}{3\delta-2}\text{ for }\delta>\frac{2}{3},\\
\left\|\frac{1}{f}\right\|\leq\frac{2}{-(1+\delta)+\sqrt{17\delta^{2}+6\delta-7}}\text{ for }\delta>\frac{-1+\sqrt{33}}{8}\simeq 0,593.
\end{gather*}
\end{thm}
At last, we deal with the case of elements of finite order.
\begin{thm}\label{skonczo}
Let $\mu\in M(G)$ satisfy $\|\mu\|\leq 1$ and $|\widehat{\mu}(\gamma)|>\delta$ for every $\gamma\in\widehat{G}$. Let
\begin{equation*}
\mu_{d}=\sum_{n=1}^{\infty}a_{n}\delta_{\tau_{n}},\text{ }|a_{1}|\geq |a_{2}|\geq\ldots.
\end{equation*}
If the order of element $\tau_{2}-\tau_{1}$ is equal to $n\in\mathbb{N}$ then
\begin{equation*}
|a_{1}|\geq \delta-\frac{1-\delta}{2(1-\sin\frac{\pi}{2n})}:=f(\delta)\text{ and }
\end{equation*}
\begin{equation*}
\|\mu^{-1}\|\leq\frac{1}{2f(\delta)-1}\text{ for }\delta>\frac{2-\sin\frac{\pi}{2n}}{3-\sin\frac{\pi}{2n}}.\\
\end{equation*}
\end{thm}
\begin{proof}
We perform the same reductions as in the proof of Theorem \ref{nies} and arrive to the following point (check (\ref{nie})):
\begin{equation*}
\left||a_{1}|+ca_{2}\gamma(\tau_{2}-\tau_{1})+\widehat{\rho}(\gamma)\right|\geq\delta\text{ for every $\gamma\in\widehat{G}$}.
\end{equation*}
As the order of $\tau_{2}-\tau_{1}$ is equal to $n$ we know, by Lemma \ref{rzedy}, that $\{\gamma(\tau_{2}-\tau_{1}):\gamma\in\widehat{G}\}$ is the set of all vertices of a regular $n$-gon inscribed in the unit circle with one vertex at the point $1$. Using elementary geometry\footnote{We estimate the distance of any point of the unit circle from the nearest vertex of the regular $n$-gon inscribed in it by the distance of a midpoint of an arc joining two neighboring vertices.} we find $\gamma_{0}\in\widehat{G}$ such that
\begin{equation*}
|c\mathrm{sgn}(a_{2})\gamma_{0}(\tau_{2}-\tau_{1})+1|\leq 2\sin\frac{\pi}{2n}.
\end{equation*}
This gives
\begin{equation*}
\left|\left||a_{1}|+ca_{2}\gamma_{0}(\tau_{2}-\tau_{1})\right|-\left||a_{1}|-|a_{2}|\right|\right|\leq \left||a_{1}|+ca_{2}\gamma(\tau_{2}-\tau_{1})-\left(|a_{1}|-|a_{2}|\right)\right|\leq 2|a_{2}|\sin\frac{\pi}{2n}
\end{equation*}
Repeating the estimates from the proof of Theorem \ref{nies} we easily get $|a_{2}|\leq\frac{1-\delta}{2(1-\sin\frac{\pi}{2n})}$, $|a_{1}|\geq\delta-\frac{1-\delta}{2(1-\sin\frac{\pi}{2n})}$ and finish the argument using Lemma \ref{latw}.
\end{proof}
\begin{rem}
It is possible to obtain a more refined bound in Theorem \ref{skonczo} applying the same method as in the proof of Theorem \ref{nies}.
\end{rem}

\subsubsection{Independent case}
We pass now to the analysis of a restricted class of measures (in fact, the discrete part is the most important) in order to get the estimates for the inverse. The first (trivial) lemma will help us to deal with measures with the discrete part consisting of a finite sum of Dirac deltas at points constituting an independent set.
\begin{lem}\label{pocz}
Let $x_{1},x_{2},\ldots,x_{N}$, $N\in\mathbb{N}$ be a finite sequence of non-negative numbers satisfying $x_{1}\geq x_{2}\geq\ldots\geq x_{N}$ and $\sum_{n=1}^{N}x_{n}\leq 1$ and let $\delta>\frac{1}{2}$. Assume also that the following inequalities hold true:
\begin{equation*}
\left|\sum_{n=1}^{N}(-1)^{n-1}x_{n}\right|\geq\delta\text{ and }\left|x_{1}-\sum_{n=2}^{N}x_{n}\right|\geq\delta
\end{equation*}
Then
\begin{equation*}
x_{1}\geq \delta+\sum_{n=2}^{N}x_{n}.
\end{equation*}
\end{lem}
\begin{proof}
As the sequence $(x_{n})_{n=1}^{N}$ is non-increasing we have $\sum_{n=1}^{N}(-1)^{n-1}x_{n}\geq 0$. Of course, $x_{1}\geq \sum_{n=1}^{N}(-1)^{n-1}x_{n}\geq\delta$. Since $\sum_{n=1}^{N}x_{N}\leq 1$ the second inequality proves the assertion.
\end{proof}
We also need a generalization of the classical Kronecker's approximation theorem (see 5.1.3 in \cite{r}). By definition, a subset $E\subset G$ is independent if for every choice of distinct points $x_{1},\ldots,x_{k}$ of $E$ and integers $n_{1},\ldots,n_{k}$ either
\begin{equation*}
n_{1}x_{1}=n_{2}x_{2}=\ldots=n_{k}x_{k}=0
\end{equation*}
or
\begin{equation*}
n_{1}x_{1}+n_{2}x_{2}+\ldots+n_{k}x_{k}\neq 0.
\end{equation*}
\begin{thm}\label{kro}
Suppose $E$ is a finite independent subset of a locally compact Abelian group $G$ consisting of elements of infinite order and let $f$ be a function on $E$ satisfying $|f(x)|=1$ for every $x\in E$. Then, for any fixed $\varepsilon>0$, there exists $\gamma\in\widehat{G}$ such that
\begin{equation*}
|\gamma(x)-f(x)|<\varepsilon\text{ for every }x\in E.
\end{equation*}
\end{thm}
Now, we are ready to prove the theorem on the norm of the inverse of a measure with a discrete part supported on an independent set of points of infinite order.
\begin{thm}\label{nornz}
Let $\mu\in M(G)$ satisfy $\|\mu\|\leq 1$, $\inf_{\gamma\in\widehat{G}}|\widehat{\mu}(\gamma)|=\delta>\frac{1}{2}$ and let the support of $\mu_{d}$ be independent and consist of elements of infinite order. Then $\|\mu^{-1}\|\leq\frac{1}{2\delta-1}$.
\end{thm}
\begin{proof}
First of all, by Corollary \ref{dysdu} we have $\inf_{\gamma\in\widehat{G}}|\widehat{\mu_{d}}(\gamma)|\geq\delta$. We will deal first with a simpler case, so let us assume that the support of $\mu_{d}$ is finite, i.e.
\begin{equation*}
\mu_{d}=\sum_{k=1}^{N}a_{k}\delta_{y_{k}}\text{ for some $N\in\mathbb{N}$, complex numbers $a_{k}$ and $y_{k}\in G$, $k=1,2,\ldots,N$}.
\end{equation*}
We can, of course, arrange the numbers $(a_{k})_{k=1}^{N}$ so that $|a_{1}|\geq |a_{2}|\geq\ldots\geq |a_{N}|$. In addition, as convolution of a measure with a Dirac delta and multiplication by a constant of modulus one do not change the norm of a measure, the norm of the inverse and do not violate the estimate $|\widehat{\mu}|\geq\delta$, we can assume, without losing the generality, the following form of $\mu_{d}$:
\begin{equation*}
\mu_{d}=a_{1}\delta_{0}+\sum_{k=2}^{N}a_{k}\delta_{y_{k}} \text{ where $a_{1}>0$}.
\end{equation*}
Let us now fix $\varepsilon>0$ with $\delta-\varepsilon>\frac{1}{2}$, any sequence of signs $(\xi_{n})_{n=2}^{N}$ and let $(\alpha_{k})_{k=2}^{N}$ be a sequence of complex numbers of modulus one such that $\alpha_{k}a_{k}=\xi_{k}|a_{k}|$ for $k=2,\ldots,N$. By Theorem \ref{kro} there is $\gamma_{0}\in\widehat{G}$ such that
\begin{equation}\label{osz}
|\gamma_{0}(y_{k})-\alpha_{k}|\leq\frac{\varepsilon}{\sum_{k=2}^{N}|a_{k}|}\text{ for $k\in\{2,\ldots,N\}$}.
\end{equation}
Now, we perform elementary estimates
\begin{gather*}
|\widehat{\mu_{d}}(\gamma_{0})|-\left|a_{1}+\sum_{k=2}^{N}|a_{k}|\xi_{k}\right|=\left|a_{1}+\sum_{k=2}^{N}a_{k}\gamma_{0}(y_{k})\right|-\left|a_{1}+\sum_{k=2}^{N}|a_{k}|\xi_{k}\right|\leq\\
\leq\left|\left(a_{1}+\sum_{k=2}^{N}a_{k}\gamma_{0}(y_{k})\right)-\left(a_{1}+\sum_{k=2}^{N}|a_{k}|\xi_{k}\right)\right|\leq\sum_{k=2}^{N}\left|a_{k}\gamma_{0}(y_{k})-|a_{k}|\xi_{k}\right|=\\
=\sum_{k=2}^{N}|a_{k}\gamma_{0}(y_{k})-a_{k}\alpha_{k}|=\sum_{k=2}^{N}|a_{k}||\gamma_{0}(y_{k})-\alpha_{k}|\leq\varepsilon\text{ by (\ref{osz})}.
\end{gather*}
Recalling that $|\widehat{\mu_{d}}(\gamma_{0})|\geq\delta$ we get
\begin{equation*}
\left|a_{1}+\sum_{k=2}^{N}|a_{k}|\xi_{k}\right|\geq\delta-\varepsilon.
\end{equation*}
This estimate holds true for every choice of signs $\xi_{k}$ and every sufficiently small $\varepsilon>0$ so we obtained, in fact, the following conclusion:
\begin{equation*}
\left|a_{1}+\sum_{k=2}^{N}|a_{k}|\xi_{k}\right|\geq\delta>\frac{1}{2}\text{ for every choice of signs $\xi_{k}$}
\end{equation*}
Now, we are in position to use Lemma \ref{pocz} which gives
\begin{equation*}
a_{1}\geq\delta+\sum_{k=2}^{N-1}|a_{k}|\geq\delta>\frac{1}{2}.
\end{equation*}
The assertion of the theorem follows now from Lemma \ref{latw}.
\\
Let us move on now to the general case, so let $\mu_{d}$ be given by the formula (we use the same reduction as before):
\begin{equation*}
\mu_{d}=a_{1}\delta_{0}+\sum_{k=2}^{\infty}a_{k}\delta_{y_{k}}\text{ where }
\end{equation*}
$a_{1}>0$, $a_{1}\geq |a_{2}|\geq |a_{3}|\geq ...$ and for every $N\in\mathbb{N}$ the set $\{y_{k}\}_{k=2}^{N}$ is independent, consists of elements of infinite order and $\sum_{n}|a_{n}|<\infty$. Put
\begin{equation*}
\nu_{N}=a_{1}\delta_{0}+\sum_{k=2}^{N}a_{k}\delta_{\tau_{k}}.
\end{equation*}
Fix $\varepsilon>0$ such that $\delta-\varepsilon>\frac{1}{2}$. Then there exists $N\in\mathbb{N}$, $N:=N(\varepsilon)$ such that $\inf_{\gamma\in\widehat{G}}|\widehat{\nu_{N}}(\gamma)|\geq\delta-\varepsilon$. By the initial part of the proof we get
\begin{equation*}
a_{1}\geq\delta-\varepsilon+\sum_{n=2}^{N}|a_{n}|.
\end{equation*}
A moment of consideration (we can assume that $N(\varepsilon)$ tends to $\infty$ as $\varepsilon\rightarrow 0$) leads to
\begin{equation*}
a_{1}\geq \sum_{n=2}^{\infty}|a_{n}|+\delta\geq\delta
\end{equation*}
which finishes the proof.
\end{proof}
The proof of the last theorem gives us a bit more.
\begin{thm}
Let $\mu\in M(G)$ be a measure satisfying the assumptions of Theorem \ref{nornz} and let $a>0$ be the mass of the greatest atom in $\mu_{d}$. Then
\begin{equation*}
\|\mu^{-1}\|\leq\frac{1}{2(\|\mu_{d}\|+\delta-a)-1}
\end{equation*}
\end{thm}
\begin{rem}
The assumptions of Theorem \ref{nornz} might seem to be quite restrictive but for the classical groups of harmonic analysis: $\mathbb{T}^{n}$ and $\mathbb{R}^{n}$ the situation covered by the theorem is a generic one.
\end{rem}
\subsubsection{Groups of exponent two}
In this section we will deal with a special case of real measures on groups of exponent two (i.e. groups such that $2x=0$ for every $x$), which has connection to the so-called \textit{dyadic analysis}. As it will be shown in Theorem \ref{rzd} we obtain the norm-controlled inversion for every $\delta>\frac{1}{2}$.
\\
We start with finite groups $\mathbb{Z}_{2}^{n}$. For convenience, we introduce the following (non-standard) ordering on $\mathbb{Z}_{2}^{n}:$ for $x,y\in\mathbb{Z}_{2}^{n}=\{0,1\}^{n}$ we define $x<y$ if the last digit of $x$ non-equal to the corresponding digit of $y$ is equal to $0$. This order is sometimes called the colexicographic order; the only difference is that we read the strings of digits from right to left. For example, for $n=3$ we have
\begin{equation*}
000<100<010<110<001<101<011<111.
\end{equation*}
The most important property of this order is that the second half of strings end with $0$, and the second half ends with $1$. What is more, if we have a string ending with $0$, it is easy to find the corresponding string that ends with $1$ -- you need to shift by $2^{n-1}$.

We will also need a precise description of the dual group of $\mathbb{Z}_{2}^{n}$ which is by definition the set of all homomorphisms $f:\mathbb{Z}_{2}^{n}\rightarrow\{-1,1\}$. It is clear that every such $f$ is uniquely determined by its values on the standard system of generators of $\mathbb{Z}_{2}^{n}$: $1000...0$, $01000...0$, ..., $000...1$ and for every choice of signs from $\{-1,1\}^{n}$ we get a homomorphism. Thus we can write that every homomorphism $f\in\widehat{\mathbb{Z}_{2}^{n}}$ is bijectively identified with an element $y\in\mathbb{Z}_{2}^{n}$ by the formula $f_{y}(x)=(-1)^{<x,y>_{n}}$ where $<x,y>_{n}$ is a formal scalar product on $\mathbb{Z}_{2}^{n}$ (for example: $<110,011>_{3}=1\cdot 0+1\cdot 1+0\cdot 1=1$). Thus, the Fourier-Stieltjes transform of $\delta_{x}$ for $x\in\mathbb{Z}_{2}^{n}$ is of the form:
\begin{equation*}
\widehat{\delta_{x}}(y)=(-1)^{<x,y>_{n}}\text{ for }y\in\mathbb{Z}_{2}^{n}.
\end{equation*}
We will use the following two symbolic operations on $\mathbb{Z}_{2}^{n}$:
\begin{enumerate}
  \item $x'$ for $x\in\mathbb{Z}_{2}^{n}$ is the element of $\mathbb{Z}_{2}^{n-1}$ obtained from $x$ by erasing the last digit.
  \item $x|0$, $x|1$ for $x\in\mathbb{Z}_{2}^{n-1}$ is the element of $\mathbb{Z}_{2}^{n}$ obtained from $x$ by inserting $0$ or $1$ (respectively) after the last digit of $x$.
\end{enumerate}

\begin{lem}\label{skondwa}
Let $n$ be a positive integer and let $\mu\in M(\mathbb{Z}_{2}^{n})$ be a real measure satisfying $\|\mu\|\leq 1$ and $|\widehat{\mu}(\gamma)|\geq\delta>\frac{1}{2}$ for every $\gamma\in\mathbb{Z}_{2}^{n}$. Then the mass of the greatest atom in the support of $\mu$ is greater or equal then $\delta$.
\end{lem}
\begin{proof}
Let
\begin{equation*}
\mu=\sum_{k=1}^{2^{n}}a_{x_{k}}\delta_{x_{k}},
\end{equation*}
where the elements $x_{k}\in\mathbb{Z}_{2}^{n}$ are ordered as explained before.

Then, using the introduced notation, our assumption can be rewritten as:
\begin{equation}\label{zaloz}
\forall_{l\in\{1,\ldots, 2^{n}\}}\left|\sum_{k=1}^{2^{n}}(-1)^{<x_{k},x_{l}>_{n}}a_{x_{k}}\right|\geq\delta.
\end{equation}
The proof will proceed by induction on $n$.

The case of $n=1$ is elementary as the inequalities $|a_{x_{1}}+a_{x_{2}}|,|a_{x_{1}}-a_{x_{2}}|\geq\delta$ for real numbers $a_{x_{1}},a_{x_{2}}$ clearly imply that the modulus of one of these numbers is greater than $\delta$.

Suppose that we have already proven the assertion for $s\leq n-1$. By the definition of our order, the last digit of $x_{l}$ is $0$ for $l\in\{1,\ldots,2^{n-1}\}$ which gives for $l\in\{1,\ldots,2^{n-1}\}$:
\begin{equation*}
<x_{l},x_{k}>_{n}=<x_{l}'|0,x_{k}'|0>_{n}=\left\{\begin{array}{c}
                                                   <x_{l}',x_{k}'>_{n-1}\text{ for }k\in\{1,\ldots,2^{n-1}\} \\
                                                   <x_{l}',x_{k-2^{n-1}}'>_{n-1}\text{ for }k\in\{2^{n-1}+1,2^{n}\}
                                                 \end{array}\right.
\end{equation*}
Inserting this to (\ref{zaloz}), we obtain
\begin{equation}\label{pr1}
\forall_{l\in\{1,\ldots,2^{n-1}\}}\left|\sum_{k=1}^{2^{n-1}}(a_{x_{k}}+a_{x_{k+2^{n-1}}})(-1)^{<x_{k}',x_{l}'>_{n-1}}\right|.
\end{equation}
Using similar arguments for elements ending with $1$, we get
\begin{equation}\label{pr2}
\forall_{l\in\{2^{n-1}+1,\ldots,2^{n}\}}\left|\sum_{k=1}^{2^{n-1}}(a_{x_{k}}-a_{x_{k+2^{n-1}}})(-1)^{<x_{k}',x_{l}'>_{n-1}}\right|.
\end{equation}
We perform a substitution $b_{x_{k}}:=a_{x_{k}}+a_{x_{k+2^{n-1}}}$, $c_{k}:=a_{x_{k}}-a_{x_{k+2^{n-1}}}$ for $k\in\{1,\ldots,2^{n-1}\}$. It is straightforward to verify that the original assertion (for $n$) reduces to two analogous problems of reduced order ($n-1$). Hence we are allowed to use the inductive hypothesis, which yields $|b_{x_{k_{1}}}|\geq\delta$ and $|c_{x_{k_{2}}}|\geq\delta$ for some $k_{1},k_{2}\in\{1,\ldots,2^{n-1}\}$. If $k_{1}\neq k_{2}$, we get a contradiction with the norm assumption: $1\geq |b_{x_{k_{1}}}|+|c_{x_{k_{2}}}|\geq 2\delta>1$; indeed, note that $|b_{x_{k_1}}| + |c_{k_2}| \leqslant |a_{x_{k_1}}| + |a_{x_{k_1+2^{n-1}}}| + |a_{x_{k_2}}| + |a_{x_{k_2 +2^{n-1}}}|$. Thus $k_{1}=k_{2}$ and it is enough to use the argument for $n=1$ to finish the proof.
\end{proof}
\begin{thm}\label{rzd}
Let $G$ be a locally compact Abelian group satisfying $2x=0$ for every $x\in G$. Let $\mu\in M(G)$ be a real measure such that $\|\mu\|\leq 1$, $|\widehat{\mu}(\gamma)|\geq\delta>0$ for every $\gamma\in\widehat{G}$. Then the mass of the greatest atom in the support of $\mu_{d}$ is greater than or equal to $\delta$. Moreover, $\mu$ is invertible and $\|\mu^{-1}\|\leq\frac{1}{2\delta-1}$.
\end{thm}
\begin{proof}
By Theorem \ref{dysk} we have $|\widehat{\mu_{d}}(\gamma)|\geq\delta$ for every $\gamma\in\widehat{G}$ and obviously $\|\mu_{d}\|\leq 1$. Let $\mu_{d}$ have the following form:
\begin{equation}\label{czdysk}
\mu_{d}=\sum_{k=1}^{\infty}a_{k}\delta_{x_{k}}\text{ for some }a_{k}\in\mathbb{R}\text{ and }x_{k}\in G.
\end{equation}
Fix $\varepsilon>0$ such that $\delta-\varepsilon>\frac{1}{2}$ and find $N_{\varepsilon}\in\mathbb{N}$ for which
\begin{equation*}
\sum_{k=N_{\varepsilon}+1}^{\infty}|a_{k}|<\varepsilon.
\end{equation*}
Define $\mu_{\varepsilon}$ as a sum as in (\ref{czdysk}) from $k=1$ to $N_{\varepsilon}$. Then clearly $\|\mu_{\varepsilon}\|\leq 1$ and $|\widehat{\mu_{\varepsilon}}(\gamma)|\geq\delta-\varepsilon>\frac{1}{2}$ for every $\gamma\in\widehat{G}$. Now it is enough to observe that $\mu_{\varepsilon}$ may be treated as an element of $M(\mathbb{Z}_{2}^{N_{\varepsilon}})$ and apply Lemma \ref{skondwa}. To finish the proof we pass with $\varepsilon$ to zero and use Lemma \ref{latw}.
\end{proof}
\section{Fourier--Stieltjes algebras}
We switch now to Fourier--Stieltjes algebras (the standard reference for this part is \cite{ey}, see also a recent monograph \cite{kl}). Let $G$ be a locally compact group and let $B(G)$ be the Fourier--Stieltjes algebra of the group $G$ i.e. the linear span of  continuous positive-definite functions on $G$ equipped with the norm given by the duality $B(G)=\left(C^{\ast}(G)\right)^{\ast}$, where $C^{\ast}(G)$ is the full group $C^{\ast}$-algebra of $G$. It is well-known that $B(G)$ with pointwise product is a commutative semisimple unital Banach algebra with $G$ embedded in $\triangle(B(G))$ via evaluation functionals. A continuous bounded function on $G$ is called (weakly) almost periodic if the set of all its translates is precompact in the (weak) uniform topology, the set of all such functions will be denoted by $AP(G)$ ($WAP(G)$). A classical result of Eberlein (see \cite{e} for the original work and \cite{gl} for a different proof) states that $B(G)\subset WAP(G)$. Moreover, there exists an invariant mean $M\in \left(WAP(G)\right)^{\ast}$ (the existence of the mean and related matters are discussed in detail in the book \cite{b}). Restricting this mean to $B(G)$ we obtain a decomposition $B(G)=B_{c}(G)\oplus (B(G)\cap AP(G))$ where $B_{c}(G)=\{f\in B(G):M(|f|)=0\}$ is a closed ideal in $B(G)$ and $B(G)\cap AP(G)$ is a closed subalgebra. Here $B_{c}(G)$ is the analogue of the ideal of continuous measures and $B(G)\cap AP(G)$ is the analogue of the subalgebra of discrete measures for commutative groups. The investigations on Fourier--Stieltjes algebras are natural generalizations of those concerning measure algebras, as in case of an Abelian group $G$ we have a canonical isomorphism of $B(G)$ with $M(\widehat{G})$.

As we work already `on the Fourier transform side', the variant of Problem \ref{pier} can be rephrased in our context as follows.
\begin{prob}\label{dru}
Let $G$ be a locally compact group and let $f\in B(G)$ satisfy $\|f\|\leq 1$ and $\inf_{x\in G}|f(x)|=\delta>0$.  What is the minimal value of $\delta_{0}>0$ such that for any $\delta>\delta_{0}$ the element $f$ is automatically invertible? What can be said about the norm of the inverse?
\end{prob}
\subsection{Qualitative result}
The inspection of the proof of Theorem \ref{glop} shows that in order to obtain an analogue for Fourier--Stieltjes algebras (and answer the first question from Problem \ref{dru}) we need only to show the version of Theorem \ref{dysk} for Fourier-Stieltjes algebras (indeed, we have only used the basic facts from Gelfand theory, the orthogonal decomposition of a measure into continuous and discrete part and the density of $\widehat{G}$ in $\triangle(M_{d}(G))$ -- all of these facts have standard counterparts for Fourier--Stieltjes algebras). However, the original proof cannot be repeated verbatim and we use a different approach in the second part of the argument.
\begin{thm}
Let $G$ be a locally compact group and let $f\in B(G)$ have the decomposition $f=g+h$ with $g\in \left(AP(G)\cap B(G)\right)$ and $h\in B_{c}(G)$. Then $g(G)\subset\overline{f(G)}$.
\end{thm}
\begin{proof}
Without loss of generality, it is enough to show $g(e)\in \overline{f(G)}$ (the general case being obtained via considering shifted $f$). First we prove the existence of distinct group elements $(x_{n})_{n=1}^{\infty}$ satisfying
\begin{equation}\label{grp}
|h(x_{1})|<\varepsilon\text{ and }|h(x_{n}x_{j}^{-1})|<\varepsilon\text{ for $j<n$ and $n>1$}.
\end{equation}
Since $h\in B_{c}(G)$ and the mean $M$ satisfies $M(f)\geq 0$ if $f\geq 0$ we clearly have $\inf_{x\in G}|h(x)|=0$ so there is no problem with the choice of $x_{1}$. Suppose that we have already picked $x_{1},\ldots,x_{n-1}$ and consider an auxiliary function $u\in B_{c}(G)$ defined as follows:
\begin{equation*}
u(x)=\sum_{j=1}^{n-1}|h(xx_{j}^{-1})|^{2}.
\end{equation*}
As $u\in B_{c}(G)$ w also have $\inf_{x\in G}|u(x)|=0$ and we are able to choose $x_{n}$ different from $x_{1},\ldots,x_{n-1}$ or $|h(e)|<\varepsilon$ which immediately finishes the proof.
\\
Let us consider now the following set of functions in $B(G)\cap AP(G)$.
\begin{equation*}
X=\{g_{x_{n}}:n\in\mathbb{N}\}\text{ where ${g}_{x}(y)=g(xy)$ for $x,y\in G$}.
\end{equation*}
Since $g\in AP(G)$ this set is relatively compact in the uniform topology. Hence there exists a subsequence $(g_{x_{n_{k}}})_{k\in\mathbb{N}}$ which is a Cauchy sequence. It follows that for $k>N$ we have
\begin{equation*}
\|g_{x_{n_{k+1}}}-g_{x_{n_{k}}}\|_{\infty}<\varepsilon.
\end{equation*}
In particular,
\begin{equation*}
|g_{x_{n_{k+1}}}(x_{n_{k}}^{-1})-g_{x_{n_{k}}}(x_{n_{k}}^{-1})|=|g(x_{n_{k+1}}x_{n_{k}}^{-1})-g(e)|<\varepsilon.
\end{equation*}
But now,
\begin{equation*}
|f(x_{n_{k+1}}x_{n_{k}}^{-1})-g(e)|=|h(x_{n_{k+1}}x_{n_{k}}^{-1})+g(x_{n_{k+1}}x_{n_{k}}^{-1})-g(e)|\leq 2\varepsilon.
\end{equation*}
\end{proof}
For completeness, let us state a variant of Theorem \ref{glop} for Fourier-Stieltjes algebras.
\begin{thm}\label{fsa}
Let $G$ be a locally compact group and let $f\in B(G)$ satisfy $\|f\|\leq 1$ and $\inf_{x\in G}|f(x)|>\frac{1}{2}$. Then $f$ is invertible.
\end{thm}
\subsection{Quantitative results}
The aim of this short section is to show that Nikolski's estimate for $\delta > \frac{1}{\sqrt{2}}$ can be achieved for Fourier-Stieltjes algebras as well. In order to do that, we need to recall some facts about absolute continuity and $L$-spaces from \cite{ow}.
\begin{de}[{cf. \cite[Definition 2.]{ex} and \cite[Definition 4.7]{ow}}]
Let $G$ be a locally compact group and let $f,g \in B(G)$. Let $\mathrm{zs}(f)$ and $\mathrm{zs}(g)$ be the central supports of $f$ and $g$, respectively, where $f$ and $g$ are viewed as functionals on $C^{\ast}(G)$. We say that $f$ is \textbf{absolutely continuous} with respect to $g$ (denoted by $f \ll g$) if $\mathrm{zs}(f) \leq \mathrm{zs}(g)$. We call $f$ and $g$ \textbf{mutually singular} if $\mathrm{zs}(f) \mathrm{zs}(g) = 0$.
\end{de}
\begin{prop}[{cf. \cite[Proposition 4.18]{ow} and \cite[Corollary 4.19]{ow}}]
Let $X\subset B(G)$ be a closed subspace that is left and right invariant by the action of $G$. Then there exists a central projection $z \in W^{\ast}(G):= (C^{\ast}(G))^{\ast\ast}$ such that $X= z\cdot B(G)$. Consequently, we have a splitting $B(G) = X \oplus X^{\perp}$, where $X^{\perp}:= (1-z) \cdot B(G)$, which is \textbf{orthogonal}, i.e. all elements from $X^{\perp}$ are mutually singular with elements from $X$.
\end{prop}
We are now able to show the following fact.
\begin{prop}\label{Prop:jedynka}
Let $f \in B(G)$. Then there exists an orthogonal decomposition $f = c\mathds{1} + f_0$, where $f_0$ and $\mathds{1}$ are mutually singular. Moreover, $c=m(f)$, where $m$ is the unique invariant mean on $B(G)$.
\end{prop}
\begin{proof}
Since the subspace consisting of constant functions is $G$-invariant, we get an orthogonal decomposition by the previous proposition. Denote the constant appearing therein by $\varphi(f)$. By uniqueness of the decomposition, it is easy to see that $\varphi: B(G) \to \mathbb{C}$ is a linear functional that satisfies $\varphi(\mathds{1})=1$. Since the decomposition is orthogonal, we get $\|f\| = |\varphi(f)| + \|f_0\|$, so $\varphi$ is a continuous functional. Note also that $\varphi$ is $G$-invariant (from both left and right). Indeed, for any $g\in G$ the equality $f = \varphi(f) \mathds{1} + f_0$ implies that $g\cdot f = \varphi(f)\mathds{1} + g\cdot f_0$, but also $g\cdot f =\varphi(g\cdot f) \mathds{1} + (g\cdot f)_0$, hence $\varphi(g\cdot f) = \varphi(f)$ by uniqueness; in the same way we obtain invariance from the right. This is almost exactly the definition of an invariant mean, but we do not know if $\varphi$ is a positive functional. We will not show this directly; in the next lemma we will prove that there exists a unique $G$-invariant functional that attains the value $1$ at $\mathds{1}$, hence it has to coincide with the invariant mean.
\end{proof}
\begin{lem}[{see also \cite[Proposition 4.10]{ll}}]
Let $\varphi\colon B(G) \to \mathbb{C}$ be a linear functional such that $\varphi(\mathds{1})=1$ and $\varphi(f) = \varphi(g\cdot f)=\varphi(f\cdot g)$ for any $g\in G$. Then it is equal to the unique invariant mean on $B(G)$.
\end{lem}
\begin{proof}
It suffices to show that there is at most one functional $\varphi$ satisfying the assumptions of the lemma. Recall that in order to show that two such functional are equal, it suffices to show that their kernels coincide; this will force them to be proportional and the condition $\varphi(\mathds{1})=1$ will finish the proof. Therefore we have to check that the kernel of $\varphi$ is uniquely specified by the conditions spelled out in the lemma. Note that $\mathrm{ker}\varphi$ is a $G$-invariant subspace, so it is of the form $z\cdot B(G)$ for some central projection in $W^{\ast}(G)$. Let $z_{\mathrm{triv}}$ be the central projection such that $z_{\mathrm{triv}} B(G) = \mathbb{C} \mathds{1}$. We will show that $z = 1 - z_{\mathrm{triv}}$. Since $\varphi(\mathds{1})=1$, it cannot be the case that $z_{\mathrm{triv}} \leq z$. As $z_{\mathrm{triv}}$ is a minimal projection, it follows that $z \leq 1 - z_{\mathrm{triv}}$. If the inequality was strict, then the codimension of the kernel would be greater than $1$, which cannot happen for a linear funcitonal, hence $z = 1 - z_{\mathrm{triv}}$.
\end{proof}
With these tools we are able to conclude.
\begin{prop}\label{jprzez}
Let $f \in B(G)$ satisfy $\|f\|\leqslant 1$ and $|f| \geqslant \delta$, where $\delta>\frac{1}{\sqrt{2}}$. Then $f$ is invertible and $\|\frac{1}{f}\| \leqslant \frac{1}{2\delta^2 - 1}$.
\end{prop}
\begin{proof}
As $|f|\geqslant \delta$, $|f|^2 \geqslant \delta^2>\frac{1}{2}$. By Proposition \ref{Prop:jedynka} we get $|f|^2 = m(|f|^2) \mathds{1} + f_0$. Since the mean $m(f)$ is characterised as being the unique constant that can be uniformly approximated with convex combinations of translates of $f$, we get $m(|f|^2)\geq \delta^2> \frac{1}{2}$. As $\| |f|^2\| \leqslant 1$, we get that $\|f_0\| \leq 1-\delta^2< \frac{1}{2}$. This means that we can write the inverse of $|f|^2$ using the Neumann series (just like in Lemma \ref{latw}) and quickly obtain the estimate for the norm of the inverse by $\frac{1}{2\delta^2-1}$; it just remains to note that $\frac{1}{f} = \frac{\overline{f}}{|f|^2}$.
\end{proof}
The proofs contained in this note clearly suggest that the most important information concerning invertibility is contained in the discrete (or almost periodic) part. In the non-commutative setting it is possible that the almost periodic part is very simple. Namely, there is a class of groups, called \textbf{minimally almost periodic}, for which constants are the only almost periodic functions; an example of such a group is $SL_2(\mathbb{R})$. For them we get a better result.
\begin{prop}
Let $G$ be a minimally almost periodic group and let $f\in B(G)$ satisfy $\|f\|\leqslant 1$ and $|f|\geq \delta > \frac{1}{2}$. Then $f$ is invertible and $\|\frac{1}{f}\| \leq \frac{1}{2\delta-1}$.
\end{prop}
\begin{proof}
Let $f = f_{ap} + f_c$ be a decomposition of $f$ into almost periodic and continuous parts, and recall that $m(|f_c|)=0$. It follows that $m(|f|) = m(|f_{ap}|)$. As the only almost periodic functions are the constants, we get $f_{ap} = m(f)\mathds{1}$, hence $|m(f)| = m(|f|)\geq \delta$. We have an orthogonal decomposition $f = m(f)\mathds{1} + f_{c}$, where $|m(f)|\geq \delta > \frac{1}{2}$, so $\|f_{c}\|\leq 1-\delta < \frac{1}{2}$ and the usual procedure involving the Neumann series proves invertibility and provides the estimate.
\end{proof}
There are two more simple cases in which we get the same estimate for the inverse; this time there are very strong assumptions about the function.
\begin{prop}
Let $G$ be a locally compact group and let $f\in B(G)$ satisfy $\|f\|\leq 1$ and $|f|\geqslant \delta > \frac{1}{2}$. Then $f$ is invertible and $\|\frac{1}{f}\|\leqslant \frac{1}{2\delta-1}$ if one of the following conditions is satisfied:
\begin{enumerate}[{\normalfont (i)}]
\item $f$ is real and has constant sign;
\item the absolute value of $f$ is constant.
\end{enumerate}
\end{prop}
\begin{proof}
\begin{enumerate}[(i)]
\item We may assume that $f$ is positive. Recall the decomposition from Proposition \ref{Prop:jedynka}: $f=m(f)\mathds{1} + f_0$. Since $f\geqslant \delta$, we get $m(f)\geqslant \delta$, so $f = m(f)\left(\mathds{1} + \frac{f_0}{m(f)}\right)$ can be inverted using the Neumann series, which also gives the estimate.

\item In this case we have $\overline{f} f = c \geqslant \delta^2$, hence $\frac{1}{f} = \frac{\overline{f}}{c}$, which gives $\|\frac{1}{f}\| \leqslant \frac{1}{\delta^2} \leqslant \frac{1}{2\delta-1}$. Note that here the assumption $\delta>\frac{1}{2}$ is not necessary.
\end{enumerate}

\end{proof}
\section{Concluding remarks}
\begin{enumerate}
  \item The critical constant $\delta_{0}=\frac{1}{2}$ is the same as for measure algebras on `analytic' semigroups as discussed in part 2.2. of \cite{nik}.
  \item The problem of obtaining an estimate of the inverse for measures $\mu\in M(G)$ satisfying $\|\mu\|\leq 1$ and $\inf_{\gamma\in\widehat{G}}|\widehat{\mu}(\gamma)|\in\left(\frac{1}{2},\frac{1}{\sqrt{2}}\right)$ seems to be very difficult even for $G=\mathbb{T}$ and the considered measures being finite sums of Dirac deltas supported at rational points of the circle. Then the estimation of the inverse is translated to an elementary but surprisingly hard problem on bounding the entries of the inverses of circulant matrices in terms of their eigenvalues. It should be also noted that it cannot be solved by simply proving that the mass of a single atom is greater than $\delta$ as the example $\mu=\frac{1}{2}\delta_{0}+\frac{i}{2}\delta_{1}\in M(\mathbb{Z}_{2})$ shows (such examples can be also constructed for non-torsion groups such as $\mathbb{Z}$).
  \item Theorem \ref{glop} is sharp for any non-discrete locally compact Abelian group as it was discussed before but Theorem \ref{fsa} is not necessarily
   sharp even for non-compact groups as there are examples of non-compact groups for which the Wiener-Pitt phenomenon does not occur (check \cite{c}). However, this Theorem is sharp for discrete groups containing an infinite Abelian subgroup (consult \cite{ow})
\end{enumerate}
\section{Acknowledgements}
The authors would like to thank Maria Roginskaya and Michał Wojciechowski for many valuable discussions. They would also like to thank Nikolai Nikolski for careful reading of the manuscript and useful remarks. Parts of this work were completed during a visit of Mateusz Wasilewski to Chalmers University of Technology; he thanks the university and the host for a stimulating research environment.

\end{document}